\documentclass[12pt]{article}
\usepackage{amsmath,amsthm,amsfonts,amssymb,epsfig}

\newcommand{\reals}{\mathbb{R}}

\newcommand{\exponent}{\operatorname{e}}

\newcommand{\interior}{\operatorname{int}}

\newtheorem{stel}{Theorem}
\newtheorem{gevolg}{Corollary}
\newtheorem{lemma}{Lemma}
\theoremstyle{remark}

\begin{document}

\title{Stability of diffusively coupled linear systems with an invariant cone}

\author{Patrick De Leenheer\footnote{Department of Mathematics and Department of Integrative Biology, Oregon State University, Supported in part by NSF-DMS-1411853, deleenhp@math.oregonstate.edu}}

\date{}

\maketitle

\begin{abstract}
This paper concerns a question that frequently occurs in various applications: Is any diffusive coupling of stable linear systems, also stable? Although it has been known for a long time that this is not the case, we shall identify a reasonably diverse class of systems for which it is true.
\end{abstract}

{\bf Keywords}: linear systems, monotone systems, diffusive coupling, asymptotic stability.

{\bf MSC2010}: 15B48, 93D05, 93D30.

\section{Introduction}
The main motivation for this paper comes from the following question. Consider a coupled linear system:
\begin{eqnarray*}
{\dot x}&=&Ax+D(y-x)\\
{\dot y}&=&By+D(x-y),
\end{eqnarray*}
where $x$ and $y$ are in $\reals^n$, $A$ and $B$ are real $n\times n$ matrices, while $D$ is an arbitrary diagonal matrix with non-negative diagonal entries. 
In mathematical biology, these systems frequently occur when linearizing diffusively coupled patched nonlinear systems at their steady states. 
The coupling terms $D(y-x)$ and $D(x-y)$ are referred to as diffusive coupling terms. 
This stems from their analogy to Fick's law for diffusion which posits that the spatial flux of a species is proportional to the gradient of the density of the species, and oriented from regions of higher density to regions of lower density.

The aforementioned question is this: If the zero steady state of the uncoupled system (i.e. when $D=0$) is asymptotically stable, does
the steady state remain asymptotically stable for all possible matrices $D$?
It has long been known that the answer to this question is negative. For instance, assume that $A=B$. If we define two new variables $z_1$ and $z_2$ in $\reals^n$:
\begin{eqnarray*}
z_1&=&\frac{1}{2}(x+y)\\
{\dot z}_2&=&\frac{1}{2}(x-y),
\end{eqnarray*}
then the dynamics in these new variables is given by:
\begin{eqnarray*}
{\dot z}_1&=&Az_1\\
{\dot z}_2&=&(A-2D)z_2
\end{eqnarray*}
Suppose that
$$
A=B=\begin{pmatrix}-2& -3 \\ 1& 1 \end{pmatrix},\textrm{ and } D=\begin{pmatrix} 1&0\\0&d \end{pmatrix},
$$
with $d\geq 0$. Then the eigenvalues of $A=B$ have negative real part (because the trace of $A$ is negative, and its determinant is positive), but 
$$
A-2D=\begin{pmatrix}
-4& -3\\
1&1-2d
\end{pmatrix},
$$
whose determinant is negative when $0<d<1/8$. Thus, although the zero steady state of the uncoupled system is asymptotically stable, it is unstable for the coupled system when $d$ lies in this range. 

Despite yielding a negative answer to the original question, this potential destabilization phenomenon has spurred a lot of interesting subsequent work. It features in synchronization theory \cite{hale}, and 
also underlies mechanisms that induce pattern formation, as noted by Turing in 1952 in \cite{turing}. At the time this was seen as a revolutionary idea, especially in biology, because diffusion was believed to always have a stabilizing effect on biological systems. The example above shows that this is not always the case.

Instead of further exploring the consequences when destabilization occurs, one can try to restrict the classes of matrices to which 
$A$ and $B$ belong to guarantee that the question can be answered affirmatively. We shall identify particular classes of matrices for which this is indeed the case. 
%This class consists of matrices that are 
%generators of linear systems having the property that a solid cone (defined below) remains forward invariant for solutions of the linear systems, and which share a common linear Lyapunov function (also defined below) on the cone. This raises the question of when a common linear Lyapunov function exists, and depending on whether the cone is polyhedral or not (defined below), we provide necessary and sufficient, respectively sufficient conditions that are purely geometrical in nature.

\section{Preliminaries}
Throughout this paper, $C\subseteq \reals^n$ will represent a proper cone, i.e. a
non-empty, closed, convex, solid and pointed cone. More precisely, $C$ is a cone ($\alpha x \in C$ for all $\alpha \geq 0$ when $x\in C$) which is solid (i.e. its interior, $\interior(C)$, is not empty) and pointed (i.e. if both $x\in C$ and $-x\in C$, then $x=0$), and it is a closed and convex subset of $\reals^n$. 

Let $K\subseteq \reals^n$ be a non-empty convex cone. We say that $K$ is finitely generated if there exists vectors $k_1,k_2,\dots, k_p$ in $\reals^n$ (called the generators of $K$) 
such that 
$$
K=\{k\in \reals^n\, | \, k=\sum_{i=1}^p \alpha_i k_i \textrm{ for some } \alpha_i\geq 0\}.
$$
It is known, see e.g. \cite{borwein}, that a finitely generated cone in $\reals^n$ is a polyhedral set, i.e. the intersection of finitely many closed half-spaces in $\reals^n$ (A closed half-space in $\reals^n$ is a set of the form $\{x\in \reals^n \, | \, <v,x>\geq a \}$ for some nonzero vector $v$ and real number $a$, where $<.,.>$ denotes the standard inner product on $\reals^n$).
%We say that $K$ is a poyhedral cone in $\reals^n$ if it is the intersection of finitely many closed half-spaces in $\reals^n$  It is known that a non-empty convex cone in $\reals^n$ is finitely generated if and only if it is polyhedral. {\bf NEED A REFERENCE}. 
Therefore, every finitely generated cone is necessarily closed, a statement which is not immediately clear from its definition.

{\bf Examples}: The non-negative orthant cone $\reals^n_+$ is a proper, finitely generated cone in $\reals^n$ with the standard basis vectors $e_1,\dots, e_n$ of $\reals^n$ serving as its generators. An example of a proper cone in $\reals^n$ with $n>1$ that is not finitely generated is the Lorenz cone:
$$
\{x\in \reals^n\, | \,(x_1^2+\dots+x_{n-1}^2)^{1/2} \leq x_n\},
$$
also known as the ice cream cone. This terminology is obviously motivated by its appearance when $n=3$. As a final example, first consider ${\cal S}_n$, the set of real, symmetric $n\times n$ matrices, which can be identified with 
$\reals^{n(n+1)/2}$. Then the subset  ${\cal P}_n$ of ${\cal S}_n$ consisting of all positive semi-definite matrices is a proper cone in ${\cal S}_n$ see e.g. \cite{boyd}.  The interior of ${\cal }P_n$ consists of the positive definite matrices, and ${\cal P}_n$ is not finitely generated.

To every convex cone $K$ in $\reals^n$ -finitely generated or not- is associated 
the dual cone $K^*$, defined as the set of linear functionals on $\reals^n$ which are non-negative on $K$. Linear functionals on $\reals^n$ are elements of the dual space of $\reals^n$, which is denoted as $(\reals^n)^*$, and assuming that $\reals^n$ is equipped with the standard inner product $<.,.>$, the Riesz Representation Theorem implies that 
every linear functional $\lambda\in (\reals^n)^*$ can be identified with a unique vector $v$ in $\reals^n$ in the sense that $\lambda(x)=<v,x>$, for all $x\in \reals^n$. It follows that 
$K^*=\{\lambda \in (\reals^n)^*\, |\, \lambda(x)\geq 0\textrm{ for all }x\in K\}\equiv \{v\in \reals^n\, |\, <v,x> \geq 0\textrm{ for all } x\in K\}$ is a non-empty closed convex cone. 

{\bf Examples}: The three cones mentioned in the examples above, namely the orthant cone, the ice cream cone, and the cone of positive semi-definite matrices are self-dual, i.e. each coincides with its dual cone, see \cite{boyd}.

We collect further well-known facts concerning cones \cite{berman,smith-mini-review,boyd}:
\begin{lemma} \label{basics}
Let $K\subseteq \reals^n$ be a closed convex cone. Then:
\begin{enumerate}
%\item $(K^*)^*=K$. {\bf I DON'T THINK I EVER USE THIS. DELETE THIS IF SO.}
\item $K$ is pointed if and only if $K^*$ is solid.
\item  $\interior(K)=\{x \in K \, |\, \lambda(x)>0\textrm{ for all } \lambda \in K^*\setminus \{0\}\}$.
\item  $\interior(K^*)=\{\lambda \in K^*\, |\, \lambda(x)>0\textrm{ for all } x\in K\setminus \{0\}\}$.
\end{enumerate}
\end{lemma}
We shall need a few more properties about cones. Let $K_1\subseteq \reals^n$ and $K_2\subseteq \reals^n$ be convex cones. The set $K_1+K_2=\{x\in \reals^n\, |\, x=x_1+x_2,\, x_1\in K_1,\, x_2\in K_2\}$ is a convex cone, containing both $K_1$ and $K_2$. For any $X\subseteq \reals^n$, its reflection with respect to the origin is defined as $\{-x\, | \, x\in X\}$, and will be denoted as $-X$.
\begin{lemma} \label{dual}
Let $K_1$ and $K_2$ be convex cones in $\reals^n$. Then
\begin{enumerate}
\item $K_1+K_2$ is pointed if and only if $K_1$ and $K_2$ are pointed, and $K_1\cap (-K_2)=\{0\}$.
\item $(K_1+K_2)^*=K_1^*\cap K_2^*$.
\end{enumerate}
\begin{proof}
\begin{enumerate}
\item Assume that $K_1+K_2$ is pointed. Then so are $K_1$ and $K_2$ since they are subsets of $K_1+K_2$. Let $x\in K_1\cap (-K_2)$. Then $x\in K_1$ and $-x\in K_2$, and thus 
$x\in K_1+K_2$. But $K_1+K_2$ is pointed, and thus $x=0$. 

Assume that $K_1$ and $K_2$ are pointed, and $K_1\cap (-K_2)=\{0\}$. Let $x\in K_1+K_2$, such that 
$-x\in K_1+K_2$ as well. Then there exist $k_1, {\tilde k_1}$ in $K_1$, and $k_2,{\tilde k_2}$ in $K_2$ such that:
$$
x=k_1+k_2\textrm{ and } -x={\tilde k_1}+{\tilde k_2},
$$
and thus that
$$
k_1+{\tilde k_1}=-(k_2+{\tilde k_2})
$$
But $K_1\cap (-K_2)=\{0\}$, and thus $k_1+{\tilde k_1}=0=k_2+{\tilde k_2}$. Then $k_1$ and $-k_1$ belong to $K_1$, and $k_2$ and $-k_2$ belong to $K_2$. As $K_1$ and $K_2$ are pointed, this implies that $k_1={\tilde k_1}=k_2={\tilde k_2}=0$, and then also $x=0$, establishing that $K_1+K_2$ is pointed.
\item
If $\lambda \in (K_1+K_2)^*$, then $\lambda(x)\geq 0$ for all $x\in K_1+K_2$. Then $\lambda(x)\geq 0$ for all $x$ in $K_1$, and for all $x$ in $K_2$, and therefore 
$\lambda\in K_1^*\cap K_2^*$. Conversely, assume that $\lambda \in K_1^*\cap K_2^*$, hence $\lambda(x)\geq 0$ for all $x\in K_1$ and for all $x\in K_2$. This implies that 
$\lambda(x)\geq 0$ for all $x$ in $K_1+K_2$, and thus $\lambda \in (K_1+K_2)^*$.
\end{enumerate}
\end{proof}
\end{lemma}

For vector spaces $V$ and $W$ we denote the set of linear maps from V to $W$ as ${\cal L}(V,W)$; but when $V=W$ we denote ${\cal L}(V,V)$ as ${\cal L}(V)$. For any subset 
$X\subseteq V$, and $T\in {\cal L}(V,W)$, we denote the image of $X$ under $T$ as $T(X)=\{w\in W\, | \, w=Tx\textrm{ for some } x\in X\}$. 
 
The image of a nonempty closed convex cone under a linear map is easily seen to be a nonempty convex cone, but it need not be closed: 

{\bf Example}:  Let $K_1$ be the ice cream cone in $\reals^3$, $K_1=\{x\in \reals^3 \,| \, (x_1^2+x_2^2)^{1/2}
\leq x_3\}$, and $K_2$ be the finitely generated cone in $\reals^3$ with a single generator $\begin{pmatrix}0 \\ 1 \\ -1\end{pmatrix}$. Then $K_1\times K_2$ is a closed convex cone in 
$\reals^6=\reals^3\times \reals^3$. Let $T\in {\cal L}(\reals^6, \reals^3)$ be defined by $T(x_1,x_2)=x_1+x_2$ for all $(x_1,x_2)\in \reals^3 \times \reals^3$. Note that for all $\epsilon >0$:
\begin{eqnarray*}
\begin{pmatrix}
1\\ 0 \\ \epsilon
\end{pmatrix}=&\begin{pmatrix}1\\ -\frac{1}{\epsilon} \\ \frac{1}{\epsilon}+\epsilon \end{pmatrix}+&\frac{1}{\epsilon}\begin{pmatrix}0\\ 1\\ -1 \end{pmatrix}\in T(K_1\times K_2),\\
&\in K_1& \; \; \in K_2,
\end{eqnarray*}
but
$$
\begin{pmatrix}1\\0\\0 \end{pmatrix} \notin T(K_1\times K_2),
$$
and thus $T(K_1\times K_2)=K_1+K_2$ is not closed. This example therefore also shows that the sum of two closed convex cones in $\reals^n$ need not be closed. Notice that 
$\textrm{Ker}(T)\cap \left( K_1\times K_2 \right) \neq \{0\}$ because it contains the vector $\begin{pmatrix}\begin{pmatrix}0\\ -1 \\ 1 \end{pmatrix}, \begin{pmatrix} 0\\ 1\\ -1\end{pmatrix} \end{pmatrix}$. 

Below is a sufficient condition guaranteeing that the image of a closed convex cone under a linear map is closed. This condition is clearly violated in the example above. 
Further results about this problem can be found in \cite{borwein}.
\begin{lemma}\label{closed}
Let $K$ be a non-empty closed convex cone in $\reals^n$, and $T\in {\cal L}(\reals^n,\reals^m)$. If 
\begin{equation}\label{ker-condition}
\textrm{Ker}(T)\cap K =\{0\},
\end{equation}
then $T(K)$ is a non-empty closed convex cone in $\reals^m$.
\end{lemma}
\begin{proof}
That $T(K)$ is a non-empty convex cone is obvious. To see that it is closed, consider a sequence $x_j$ in $T(K)$ such that $x_j\to x$ for some $x\in \reals^m$. We need to show that $x\in T(K)$. If $x=0$ the result is clear because then $0=T0$ belongs to $T(K)$. So we assume that $x\neq 0$, and therefore, for all sufficiently large $j$, holds that $||x_j|| \neq 0$.
Moreover, there exists a sequence $k_j$ in $K$ such that $x_j=Tk_j$ for all $j$. Then for all sufficiently large $j$ holds:
\begin{equation}\label{limit-seq}
\frac{1}{||x_j||}x_j=\frac{1}{||Tk_j||} Tk_j= \frac{1}{||T(k_j/||k_j||)||} T(k_j/||k_j||) \, \to \, \frac{1}{||x||}x
\end{equation}
As $e_j=k_j/||k_j||$ belongs to $S^{n-1}\cap K$ for all large $j$, where $S^{n-1}=\{x\in \reals^n \,| \, ||x||=1\}$ denotes the unit sphere in $\reals^n$, and 
$S^{n-1}\cap K$ is compact, we can extract a converging subsequence also denoted by $e_j$, with limit $e$ in $S^{n-1}\cap K$. By $(\ref{ker-condition})$ follows that 
$||Te||>0$, and by passing through the limit in  $(\ref{limit-seq})$, that
$$
x=\frac{||x||}{||Te||}Te \in T(K).
$$ 
\end{proof}

The image of a finitely generated cone in $\reals^n$ under a continuous linear map is also finitely generated, hence a polyhedral set, and thus closed:
%Since every finitely generated convex cone in $\reals^n$ is polyhedral
%\footnote{proof: It suffices to show that every half-ray through a generator can be written as a finite intersection of half-spaces} (i.e. obtained as the intersection of finitely many closed half-spaces in $\reals^n$) 
%and thus closed, the image of a finitely generated convex cone is always closed:
\begin{lemma}\cite{borwein} \label{closed2}
Let $K$ be a finitely generated nonempty convex cone in $\reals^n$, and $T\in {\cal L}(\reals^n,\reals^m)$. Then 
$T(K)$ is a non-empty closed convex cone in $\reals^m$.
\end{lemma}
\begin{proof}
$T(K)$ is obviously a non-empty convex cone.
If $k_1,\dots, k_p$ are the generators of $K$, then every element in $T(K)$ is a linear combination of the vectors $T(k_1),\dots, T(k_n)$ with non-negative coefficients. Hence 
$T(k_1),\dots, T(k_n)$ are generators for $T(K)$. Thus, $T(K)$ is a finitely generated cone and therefore it is closed.
\end{proof}

\section{Linear Lyapunov functions}
Consider the linear system
\begin{equation}\label{linear-sys}
{\dot x}=Ax,
\end{equation}
where $x\in \reals^n$ and $A\in {\cal L}(\reals^n)$. Suppose that $C$ is a proper cone in $\reals^n$. A natural question is under what conditions on $A$, the cone $C$ is a forward invariant set for $(\ref{linear-sys})$, i.e. when is $\exponent^{tA}x_0 \in C$ for all $t>0$, whenever $x_0\in C$. The answer to this question is known, see for instance \cite{berman,smith-mini-review} and references therein. We say that $A$ is quasi-monotone for $C$ (QM for short) if 
\begin{equation}\label{QM}
\textrm{Whenever } (x,\lambda)\in \partial C\times C^* \textrm{ is such that } \lambda(x)=0,\textrm{ then } \lambda(Ax)\geq 0.
\end{equation}
Here, $\partial C$ denotes the boundary of $C$. There holds that
\begin{stel} \cite{berman,smith-mini-review} \label{invariance}
Let $C$ be a proper cone in $\reals^n$, and $A\in {\cal L}(\reals^n)$. Then $C$ is a forward invariant set for $(\ref{linear-sys})$ if and only if $A$ is QM for $C$.
\end{stel}
{\bf Examples}: 
It is well-known, see \cite{smith-mini-review}, that when $C=\reals^n_+$, an $n\times n$ matrix $A$ is QM on $C$ if and only if $A_{ij}\geq 0$ for all $i\neq j$.

It was shown in \cite{wolkowicz} that when $C$ is the ice cream cone in $\reals^n$, then $A\in \reals^{n\times n}$ is QM on $C$ if and only if there exists $\alpha \in \reals$ such that:
$$
QA+A^TQ+\alpha Q
$$
is a negative semi-definite matrix. Here, $Q$ is the diagonal matrix with $Q_{ii}=1$ for all $i=1,\dots, n-1$, and $Q_{nn}=-1$. 

Suppose that $n=3$, and let 
$C$ be the ice cream cone $\{x\in \reals^3\, | \,(x_1^2+x_{2}^2)^{1/2} \leq x_3\}$ in $\reals^3$. Suppose that $\epsilon_1$ and $\epsilon_2$ are parameters, and let:
$$
A=\begin{pmatrix}
-\epsilon_1 &-1& 0\\
1& -\epsilon_1 &0\\
0&0& -\epsilon_2
\end{pmatrix}.
$$
Then $A$ is QM on $C$ if and only if
$$
\epsilon_2\leq \epsilon_1.
$$
To see this, note that:
$$
QA+A^TQ+\alpha Q=\begin{pmatrix}
-2\epsilon_1 + \alpha &0&0\\
0&-2\epsilon_1 + \alpha & 0\\
0&0& 2\epsilon_2 -\alpha
\end{pmatrix}
$$
which is negative semi-definite for some $\alpha \in \reals$, provided that $2\epsilon_2 \leq \alpha \leq 2\epsilon_1$, for some $\alpha \in \reals$. But  this is equivalent to $\epsilon_2\leq \epsilon_1$, as claimed.

{\bf Definition}: Let $C$ be a proper cone in $\reals^n$, and suppose that $A\in {\cal L}(\reals^n)$ is QM for $C$. Then $\lambda \in C^*$ is said to be a linear Lyapunov function for $(\ref{linear-sys})$ 
on $C$ if:
\begin{enumerate}
\item $\lambda(c)>0$ for all $c\in C\setminus \{0\}$.
\item $\lambda(Ac)<0$ for all $c\in C\setminus \{0\}$.
\end{enumerate}
It follows readily from Lyapunov's stability Theorem, that if $\lambda$ is a linear Lyapunov function on $C$, then the zero steady state of $(\ref{linear-sys})$ is asymptotically stable with respect to initial conditions in $C$. In fact, below we will show that a stronger conclusion holds. We say that $A\in {\cal L}(\reals^n)$ is Hurwitz if all the eigenvalues of $A$ have negative real part. 
It is well-known that $A$ is Hurwitz if and only if the zero steady state of system $(\ref{linear-sys})$ is asymptotically stable with respect to initial conditions in $\reals^n$.
\begin{stel}\label{L-H}
Let $C$ be a proper cone in $\reals^n$, and suppose that $A\in {\cal L}(\reals^n)$ is QM for $C$.
There exists a linear Lyapunov function for $(\ref{linear-sys})$ on $C$ if and only if the zero steady state of $(\ref{linear-sys})$ is asymptotically stable with respect to all initial conditions in $\reals^n$.
\end{stel}
\begin{proof}
{\it Necessity}: Suppose that there exists a linear Lyapunov function for $(\ref{linear-sys})$ on $C$. We need to prove that every solution $x(t)$ of $(\ref{linear-sys})$ in $\reals^n$ converges to $0$ as $t\to \infty$.

Since $C$ is solid, we can pick $x_0\in \interior(C)$. Set $U=\textrm{span}\{x_0\}$. Pick a basis for $U^\perp$, say $x_1,\dots ,x_{n-1}$, and note that 
$x_0,x_1,\dots, x_{n-1}$ is a basis for $\reals^n$ because $\reals^n=U\oplus U^\perp$.  
Since $x_0 \in \interior(C)$, we can 
pick $\epsilon>0$ sufficiently small, such that $x_0+\epsilon x_i\in \interior(C)$ for all $i=1,\dots, n-1$. We claim that 
$x_0,x_0+\epsilon x_1,\dots, x_0 +\epsilon x_{n-1}$ is a basis of $\reals^n$ which is clearly contained in $\interior(C)$. To prove the claim, let $\alpha_0,\dots \, ,\alpha_{n-1}$ be real scalars such that: 
$$
\alpha_0 x_0+\alpha_1(x_0+ \epsilon x_1)+\dots \alpha_{n-1}(x_0+\epsilon x_{n-1})=0,
$$
or equivalently: 
$$
\left(\sum_{i=0}^{n-1} \alpha_i\right) x_0 + \alpha_1 \epsilon x_1+\dots +\alpha_{n-1}\epsilon x_{n-1}=0.
$$
Then as $\epsilon>0$, $\alpha_0=\alpha_1=\dots =\alpha_{n-1}=0$ because $x_0,x_1,\dots, x_{n-1}$ is a basis for $\reals^n$. This proves the claim. 
We can now define a fundamental matrix solution for $(\ref{linear-sys})$ (i.e. an $n\times n$ matrix whose columns are solutions of $(\ref{linear-sys})$ that are linearly independent for all $t$), namely:
$$
X(t)=\left[ x_0(t) \;x_1(t)\, \dots \; x_{n-1}(t)\right].
$$
Here, the columns $x_0(t),x_1(t),\dots , x_{n-1}(t)$ are the unique solutions of $(\ref{linear-sys})$ with respective initial conditions $x_0$, $x_0+\epsilon x_1,\dots, x_0+\epsilon x_{n-1}$. 
By Theorem  $\ref{invariance}$, every solution $x_i(t)$ belongs to $C$ for all $t\geq 0$. And since there is a  linear Lyapunov function for $(\ref{linear-sys})$ on $C$, it follows that 
$\lim_{t\to \infty} x_i(t)=0$ for all $i=0,\dots, n-1$. But every solution of $(\ref{linear-sys})$ on $\reals^n$ is a linear combination of the columns of $X(t)$, and therefore every solution of $\reals^n$ converges to $0$ as well. This concludes the proof of this part of the Theorem.

{\it Sufficiency}: If $A$ is Hurwitz, it follows upon integration  from $t=0$ to $\infty$ of the identity: $d/dt (\exponent^{tA})=A \exponent^{tA}$ for all $t$,  that
$$
-I=A\int_0^{\infty} \exponent^{tA}dt,
$$
and thus that
$$
-A^{-1}=\int_0^{\infty} \exponent^{tA}dt.
$$
Since $A$ is QM on $C$, Theorem $\ref{invariance}$ implies that $-A^{-1}\in {\cal L}(\reals^n)$ satisfies:
$$
-A^{-1}(C) \subseteq C.
$$
Then the dual of $-A^{-1}$, denoted by $(-A^{-1})^*\in {\cal L}((\reals^n)^*)$, and equal to $(-A^*)^{-1}$, satisfies:
$$
(-A^*)^{-1}(C^*) \subseteq C^*.
$$
We claim that:
\begin{equation}\label{essence}
(-A^*)^{-1}(\interior (C^*)) \cap \interior(C^*) \neq \emptyset.
\end{equation}
From $(\ref{essence})$ follows that there exist $\lambda \in \interior (C^*)$ and ${\tilde \lambda} \in \interior(C^*)$ such that $(-A^*)^{-1}{\tilde \lambda}=\lambda$. Therefore, using 
Theorem $\ref{basics}$, there holds that:
\begin{enumerate}
\item $\lambda (c) >0$ for all $c\in C\setminus \{0\}$.
\item $\lambda(Ac)=(A^* \lambda)(c)=-{\tilde \lambda}(c)<0$ for all $c\in C \setminus \{0\}$.
\end{enumerate}
Thus, $\lambda$ is a linear Lyapunov function for $(\ref{linear-sys})$ on $C$.

To prove $(\ref{essence})$, first note that $C^*$ is solid by Lemma $\ref{basics}$ (because $C$ is a proper cone, hence pointed). Pick $c^*\in \interior(C^*)$, and let 
$U\subseteq \interior(C^*)$ be an open set such that $c^*\in U$.  By the Open Mapping  Theorem, 
$(-A^*)^{-1}(U)$ is open in $(\reals^n)^*$, and it is contained in $C^*$ because $(-A^*)^{-1}(C^*) \subseteq C^*$. If $(-A^*)^{-1}(c^*) \in \interior(C^*)$, then 
$(\ref{essence})$ is immediate because $c^*$ belongs to the intersection. If $(-A^*)^{-1}(c^*) \in \partial C^*$, then $(-A^*)^{-1}(U) \cap \interior(C^*)\neq \emptyset$, and $(\ref{essence})$ follows as well. This establishes $(\ref{essence})$, and concludes the proof.
\end{proof}

{\bf Example}: 
Suppose that $n=3$, and let $C$ be the ice cream cone $\{x\in \reals^3\, | \,(x_1^2+x_{2}^2)^{1/2} \leq x_3\}$ in $\reals^3$. 
We have seen in an example above that if
$$
A=\begin{pmatrix}
-\epsilon_1 &-1& 0\\
1& -\epsilon_1 &0\\
0&0& -\epsilon_2
\end{pmatrix},
$$
where $\epsilon_1$ and $\epsilon_2$ are real parameters, then 
$A$ is QM on $C$ if and only if
$$
\epsilon_2\leq \epsilon_1.
$$
Note that $A$ is QM on $C$ and Hurwitz if and only if :
$$
0<\epsilon_2\leq \epsilon_1,
$$
and that in this case, choosing 
$$
\lambda(x)=x_3,
$$
yields that $\lambda(Ax)=-\epsilon_2 x_3$. Thus,  $\lambda$ is a linear Lyapunov function for $(\ref{linear-sys})$ on $C$.

\section{Common linear Lyapunov functions}
Consider a linear time-varying system
\begin{equation}\label{switched-sys}
{\dot x}=A(t)x,
\end{equation}
where $x\in \reals^n$ and $A(t):\reals_+ \to {\cal L}(\reals^n)$ is a piecewise continuous map. 
We shall denote the unique solution at any time $t\geq t_0$, starting in $x_0$ at time $t_0\geq 0$ by $x(t,t_0,x_0)$.

Suppose that $C$ is a proper cone in $\reals^n$. We say that $A(t)$ is quasi-monotone for $C$ (QM for short) if: 
\begin{equation}\label{QM2}
\textrm{For all }t\in \reals_+, \textrm{whenever } (x,\lambda)\in \partial C\times C^* \textrm{ is such that } \lambda(x)=0,\textrm{ then } \lambda(A(t)x)\geq 0.
\end{equation}

There holds that:
\begin{stel} \cite{smith-mini-review} \label{invariance2}
Let $C$ be a proper cone in $\reals^n$, and $A(t): \reals_+ \to {\cal L}(\reals^n)$ a piecewise continuous map. Then for all $x_0 \in C$, the solution $x(t,t_0,x_0)$ of $(\ref{switched-sys})$ belongs to 
$C$ for all $t\geq t_0$, and for all $t_0\geq 0$,  if and only if $A(t)$ is QM for $C$.
\end{stel}

We shall be mainly interested in the behavior of solutions of the system $(\ref{switched-sys})$ in 
the case where $A(t): \reals_+\to \{A_1,\dots, A_m\}$ is an arbitrary piecewise constant map, and $\{A_1,\dots, A_m\}$ is a fixed, finite collection of linear operators on $\reals^n$. 
In the engineering literature, a system of this form is referred to as a switched system \cite{gurvits,valcher}, although strictly speaking we are dealing with a collection of systems, one for each choice of $A(t)$.

Theorem $\ref{invariance2}$ then implies:
\begin{gevolg} Let $C$ be a proper cone in $\reals^n$, and let ${\cal A}=\{A_1,\dots, A_m\}$ be a finite collection of linear operators on $\reals^n$.
Then every solution of every system $(\ref{switched-sys})$, where $A(t):\reals_+ \to {\cal A}$ is an arbitrary piecewise constant map, remains in $C$ for all $t\geq t_0$, for all $t_0\geq 0$, and for all $x_0 \in C$, if and only if $A_i$ is QM for $C$ for all $i=1,\dots, m$.
\end{gevolg}

{\bf Definition}: Let $A_1,A_2,\dots, A_m \in {\cal L}(\reals^n)$. 
Let $C$ be a proper cone in $\reals^n$, and suppose that $A_i$ is QM for $C$ for all $i=1,\dots, m$. 
Then $\lambda \in C^*$ is said to be a common linear Lyapunov function for $\{A_1,\dots, A_m\}$ on $C$, if:
\begin{enumerate}
\item $\lambda(c)>0$ for all $c\in C\setminus \{0\}$.
\item $\lambda(A_ic)<0$ for all $c\in C\setminus \{0\}$, and all $i=1,\dots, m$.
\end{enumerate}
It follows from Lyapunov's stability Theorem that if $A_1,\dots ,A_m$ have a common linear Lyapunov function on $C$, then the zero steady state of system 
$(\ref{switched-sys})$ where $A(t):\reals_+ \to \{A_1,\dots ,A_m\}$ is an arbitrary piecewise constant map, 
is uniformly asymptotically stable with respect to initial conditions in $C$. A stronger conclusion is as follows:

\begin{stel}\label{common-L}
 Let ${\cal A}=\{A_1,A_2,\dots, A_m\} \subset {\cal L}(\reals^n)$, let $C$ be a proper cone in $\reals^n$, and suppose that $A_i$ is QM for $C$ for all $i=1,\dots, m$.
If there exists a common linear Lyapunov function for ${\cal A}$ on $C$, then the zero steady state of system  $(\ref{switched-sys})$ where $A(t):\reals_+ \to {\cal A}$ is an arbitrary piecewise constant map, is uniformly asymptotically stable with respect to all initial conditions in $\reals^n$. 
\end{stel}
\begin{proof}
The proof is similar to the {\it Necessity} part of the proof of Theorem $\ref{L-H}$.
\end{proof}

The converse statement in Theorem $\ref{common-L}$ does not hold, as the following example shows:

{\bf Example}: Let $C=\reals^2_+$, and 
$$
A_1=\begin{pmatrix}
-1& 0\\
1&-1
\end{pmatrix},\textrm{ and } 
A_2=\begin{pmatrix}
-1& 1\\ 
0& -1 \end{pmatrix}.
$$
Note that $A_1$ and $A_2$ are QM on $C$, and both are Hurwitz. Therefore, 
Theorem 3.2 in \cite{gurvits} establishes that the zero steady state of  system $(\ref{switched-sys})$ where $A(t):\reals_+ \to \{A_1,A_2\}$ is an arbitrary piecewise constant map, 
is uniformly asymptotically stable with respect to all initial conditions in $\reals^2$ if and only if $A_1A_2^{-1}$ has no negative eigenvalues. Here, 
$$
A_1A_2^{-1}=\begin{pmatrix}
-1& 0\\
1&-1
\end{pmatrix}
\begin{pmatrix}
-1&-1\\
0&-1
\end{pmatrix}=
\begin{pmatrix} 
1& 1\\
-1&0
\end{pmatrix}
$$ 
and this matrix has no negative eigenvalues (in fact, it has no real eigenvalues). 
But there is no common linear Lyapunov function for $\{A_1,A_2\}$ on $C$. Indeed, suppose that $(v_1,v_2) \in \interior(C^*)= \interior(\reals^2_+)$ (recall that $C=\reals^2_+$ is self-dual), is such that
\begin{eqnarray*}
(v_1, v_2)A_1&=&(-v_1+v_2,-v_2)\in -\interior(\reals^2_+),\textrm{ and }\\
(v_1, v_2)A_2&=&(-v_1, v_1-v_2)\in -\interior(\reals^2_+).
\end{eqnarray*}
In particular, $v_1-v_2<0$ and $v_2-v_1<0$ must hold simultaneously, which is impossible.

\section{When does a common Lyapunov function exist?}
Theorems $\ref{L-H}$ and $\ref{common-L}$  motivate the search for conditions that characterize when a finite collection of linear operators that are QM on a cone, share a common Lyapunov function.

We shall consider the $(m+1)$-fold Cartesian product of $\reals^n$ with itself, $\reals^n\times \dots \times \reals^n$, and denote it as $(\reals^n)^{m+1}$. For any subset $X$ of 
$\reals^n$, the notation $X^{m+1}$ is defined similarly. For a given collection of linear operators $A_1,A_2,\dots ,A_m$ in ${\cal L}(\reals^n)$, we consider the map $T\in {\cal L}((\reals^n)^{m+1},\reals^n)$ defined by 
$$
T(x_0,x_2,\dots,x_m)=x_0-A_1x_1-A_2x_2-\dots-A_mx_m,
$$
for each $(x_0,x_1,\dots,x_m)\in (\reals^n)^{m+1}$. Then we have that:
\begin{stel} \label{general-cone}
Let $C$ be a proper cone in $\reals^n$, and $A_1,A_2,\dots, A_m\in {\cal L}(\reals^n)$ be QM on $C$. If 
\begin{equation}\label{sufficient}
\textrm{Ker}(T)\cap C^{m+1} =\{0\},
\end{equation}
then $A_1,A_2,\dots,A_m$ have a common linear Lyapunov function on $C$.
\end{stel}
\begin{proof}
If $(\ref{sufficient})$ holds then we claim that:
\begin{enumerate}
\item $-(A_1(C)+A_2(C)+\dots + A_m(C))$ is a closed convex cone.
\item $C\cap \left(A_1(C)+A_2(C)+\dots + A_m(C)\right)=\{0\}$.
\item $-(A_1(C)+A_2(C)+\dots + A_m(C))$ is pointed.
\end{enumerate}
1. follows from Lemma \ref{closed} because $C^{m+1}$ is a nonempty closed convex cone in $(\reals^n)^{m+1}$.  
To prove 2., pick $c_0\in C\cap \left(A_1(C)+A_2(C)+\dots + A_m(C)\right)$. Then for all $i=1,\dots, m$, there exist $c_i\in C$  such that $c_0=A_1c_1+\dots A_mc_m$, and thus 
$(c_0,c_1,\dots, c_m)\in \textrm{Ker}(T)\cap C^{m+1} =\{0\}$, which implies that $c_0=0$. To prove 3., it suffices to prove that 
$A_1(C)+A_2(C)+\dots + A_m(C)$ is pointed. Suppose that $x\in A_1(C)+A_2(C)+\dots + A_m(C)$, such that 
$-x \in A_1(C)+A_2(C)+\dots + A_m(C)$. Then for all $i=1,\dots, m$ there exist $c_i\in C$ and ${\tilde c}_i\in C$ such that:
\begin{eqnarray*}
x&=&A_1c_1+\dots A_mc_m\\
-x&=&A_1{\tilde c}_1+\dots A_m{\tilde c_m}
\end{eqnarray*}
and therefore
$$
0+A_1(c_1+{\tilde c}_1)+\dots A_m(c_m +{\tilde c}_m)=0.
$$
Since $C$ is a convex cone, this implies that $(0,c_1+{\tilde c}_1,\dots, c_m+{\tilde c}_m)\in \textrm{Ker}(T)\cap C^{m+1}=\{0\}$, and thus both $c_i\in C$ and $-c_i \in C$ for all $i=1,\dots, m$. 
Since $C$ is pointed, it follows that $c_i={\tilde c}_i=0$ for all $i=1,\dots, m$, and therefore that $x=0$.

From 1.,2, and 3. and  Lemma $\ref{dual}$ follows that 
$$
C-\left(A_1(C)+A_2(C)+\dots + A_m(C)\right)
$$
is a closed, pointed convex cone, hence its dual cone, which by Lemma $\ref{dual}$ equals 
$$
C^*\cap(-A_1(C))^*\cap \dots (-A_m(C))^*
$$
is solid, or equivalently that 
\begin{equation}\label{nonvoid}
\interior(C^*)\cap \interior((-A_1(C))^*)\cap \dots \cap \interior((-A_m(C))^*)\neq \emptyset.
\end{equation}
Notice that $(\ref{sufficient})$ implies that
$$
\textrm{Ker}(-A_i)\cap C=\{0\}\textrm{ for all } i=1,\dots, m,
$$
and thus $-A_i(C)$ is a closed convex cone for all $i=1,\dots,m$ by Lemma $\ref{closed}$. Then from Lemma $\ref{basics}$ follows that for all $i=1,\dots, m$:
\begin{eqnarray*}
\interior((-A_i(C))^*)&=&\{\lambda \in  (\reals^n)^*\,| \, \lambda(x)>0\textrm{ for all } x\in -A_i(C)\setminus \{0\}\} \\
&=& \{\lambda \in  (\reals^n)^*\, |\, \lambda(-A_ic)>0\textrm{ for all }c\in C\setminus \{0\}\}
\end{eqnarray*}
which together with $(\ref{nonvoid})$ implies that $A_1,\dots,A_m$ have a common Lyapunov function on $C$. 
\end{proof}
When  $C$ is a finitely generated proper cone in $\reals^n$, we have a necessary and sufficient condition for the existence of a linear common Lyapunov function on $C$:
\begin{stel} \label{finitely-generated cone}
Let $C$ be a finitely generated proper cone in $\reals^n$, and $A_1,A_2,\dots, A_m\in {\cal L}(\reals^n)$ be QM on $C$. 

Then $A_1,A_2,\dots,A_m$ have a common linear Lyapunov function on $C$ if and only if the following conditions hold:
\begin{enumerate}
\item $\textrm{Ker}(A_i)\cap C=\{0\}$ for all $i=1,\dots,m$.
\item $A_1(C)+A_2(C)+\dots +A_m(C)$ is pointed.
\item $C\cap (A_1(C)+A_2(C)+\dots +A_m(C))=\{0\}$.
\end{enumerate}
\end{stel}
\begin{proof}
{\it Sufficiency}: We will verify that $(\ref{sufficient})$ holds. Let $(c_0,c_1,\dots,c_m)\in \textrm{Ker}(T)\cap C^{m+1}$. Then
$$
c_0-A_1c_1-A_2c_2-\dots -A_mc_m=0, \textrm{ and } c_i\in C\textrm{ for all }i=0,1,\dots,m.
$$
Then $c_0=0$ by 3., and thus: 
\begin{equation}\label{sum}
0=A_1c_1+A_2c_2+\dots +A_mc_m.
\end{equation}
If $c_i\neq 0$ for some $i\in \{1,\dots,m\}$, then $Ac_i\neq 0$ by 1. Moreover, by $(\ref{sum})$:
$$
-A_ic_i=\sum_{j\neq i}^m A_jc_j \in A_1(C)+A_2(C)+\dots +A_m(C).
$$
Thus, $A_ic_i$ and $-A_ic_i$ are non-zero and belong to $A_1(C)+A_2(C)+\dots +A_m(C)$, contradicting 2. Thus $c_i=0$ for all $i=1,\dots, m$, and then this part of the proof is concluded by applying Theorem $\ref{general-cone}$.

{\it Necessity}: Suppose that $\lambda \in C^*$ is a common linear Lyapunov function for $A_1,A_2,\dots, A_m$ on $C$. Then
\begin{equation}\label{lyap}
\lambda(c)>0\textrm{ and } \lambda(-A_ic)>0\textrm{ for all }c\in C\setminus \{0\} \textrm{ and all } i=1,\dots, m.
\end{equation}
Then 1. must hold, for if it did not, there would exist some $i\in \{1,\dots,m\}$ and some $c\in C\setminus \{0\}$ such that $A_ic=0$, whence $\lambda (-A_ic)=0$, a contradiction. Note that since $1.$ holds, $(\ref{lyap})$ is equivalent to the statement that:
$$
\lambda \in \interior(C^*)\cap \left(\cap_{i=1}^m \{\lambda \in (\reals^n)^* \,| \, \lambda(-A_ic) >0\textrm{ for all } c\in C\setminus \{0\}:-A_ic\neq 0\} \right)
$$
Since $-A_i(C)$ is a non-empty closed (by Lemma $\ref{closed2}$) convex cone for all $i=1,\dots , m$, Lemma $\ref{basics}$ implies that the latter is equivalent to:
\begin{eqnarray*}
\lambda &\in& \interior(C^*)\cap \interior((-A_1(C))^*)\cap \dots \cap \interior((-A_m(C))^*)\\
&=&\interior \left(C^*\cap (-A_1(C))^*\cap \dots \cap (-A_m(C))^* \right),
\end{eqnarray*}
and thus $C^*\cap (-A_1(C))^*\cap \dots \cap (-A_m(C))^*=\left(C-(A_1(C)+\dots +A_m(C)) \right)^*$ (by Lemma $\ref{dual}$) is a closed  solid cone. 
Note that it is the dual of the closed cone $C-(A_1(C)+\dots +A_m(C))$ (closedness follows from Lemma $\ref{closed2}$), which in turn must be 
pointed (by Lemma $\ref{basics}$). Lemma $\ref{dual}$ then implies that 2. and 3. hold, concluding this part of the proof. 
\end{proof}
In the special case where $C=\reals^n_+$, and $A_1,\dots, A_m$ are QM on $\reals^n_+$, different characterizations for the existence of a 
common linear Lyapunov function for $A_1,\dots, A_m$ on $\reals^n_+$ can be found in \cite{valcher}.

{\bf Examples}: We shall first provide some examples that show that no pair of the three conditions in Theorem $\ref{finitely-generated cone}$ implies the third, indicating that these conditions are sharp for the existence of a common linear Lyapunov function when the cone is finitely generated.

When $C=\reals^2_+$, the matrices
$$
A_1=\begin{pmatrix}
-1& 0\\
2&1
\end{pmatrix}\textrm{ and } 
A_2=\begin{pmatrix}
1& 2\\ 
0& -1 \end{pmatrix}.
$$
are QM for $C$ and invertible. Thus, 1. holds. Since
$$
A_1(C)+A_2(C)\textrm{ is finitely generated by } \begin{pmatrix}-1\\2 \end{pmatrix}\textrm{ and }\begin{pmatrix}2\\-1 \end{pmatrix},
$$
it is a pointed cone, so 2. holds as well. However, 3. fails because $\begin{pmatrix}1\\1 \end{pmatrix}$ is contained in the intersection of $A_1(C)+A_2(C)$ and $C$. Thus $A_1$ and $A_2$ do not share a common linear Lyapunov function on $C$.
%Let's return to an example discussed earlier with $C=\reals^2_+$, and 
%$$
%A_1=\begin{pmatrix}
%-5& 4\\
%1&-1
%\end{pmatrix},\textrm{ and } 
%A_2=\begin{pmatrix}
%-1& 1\\ 
%4& -5 \end{pmatrix}.
%$$
%We already established in a direct -and perhaps somewhat unsatisfactory -way that $A_1$ and $A_2$ do not have a common linear Lyapunov function on $C$. Instead, we can now use  the more appealing characterization in Theorem $\ref{finitely-generated cone}$, and check which of the 3 conditions fail. Clearly, 1. holds since $\textrm{Ker}(A_1)=\textrm{Ker}(A_2)=\{0\}$. But 2. and 3. fail because
%$$
%A_1(C)+A_2(C)=\reals^2,
%$$
%which is not pointed, and $C\cap \reals^2=C \neq \{0\}$.

The matrices
$$
B_1=B_2=\begin{pmatrix}
-1&1\\
1&-1
\end{pmatrix}
$$
are QM for $C$. Here 1., fails because $\begin{pmatrix}1\\1 \end{pmatrix}$ is contained in $\textrm{Ker}(B_1)$ and in $C$, although 
$$
B_1(C)+B_2(C) =\textrm{span}\bigg\{ \begin{pmatrix}-1\\1 \end{pmatrix} \bigg\},
$$ 
is a pointed cone which intersects $C$ only in $0$, hence 2. and 3. hold. $B_1$ and $B_2$ do not share a common linear Lyapunov function on $C$.

When $C=\reals^2_+$, the matrices
$$
E_1=\begin{pmatrix}
-1&0\\
1&-1
\end{pmatrix}\textrm{ and } E_2=\begin{pmatrix}-1&1\\0&-1 \end{pmatrix}
$$
are QM for $C$, and they are invertible. Thus 1. holds. Note that 
$$
E_1(C)+E_2(C)= \{x\in \reals^2\, | \, x_1+x_2 \leq 0\}.
$$ 
Although $E_1(C)+E_2(C)$ only intersects $C$ in $0$ (so that 3. holds), this cone is not pointed, so 2. fails.

To end on a positive note, we give an example where a common linear Lyapunov function does exist on $C=\reals^2_+$. Let
$$
F_1=\begin{pmatrix}
-2&0\\
1&-1
\end{pmatrix}\textrm{ and } 
F_2=
\begin{pmatrix}
-1&1\\
0&-2
\end{pmatrix},
$$
which are QM for $C$, and invertible. Thus 1. holds. Moreover, 
$$
F_1(C)+F_2(C)\textrm{ is finitely generated by } \begin{pmatrix}-2\\1 \end{pmatrix} \textrm{ and } \begin{pmatrix}1\\-2 \end{pmatrix},
$$
and it is a pointed cone, which only intersects $C$ in 0. Thus, 2. and 3. hold as well, and therefore $F_1$ and $F_2$ share a common linear Lyapunov function. For instance, 
$\lambda(x)=x_1+x_2$ is easily seen to be a common linear Lyapunov function on $C$.

\section{Diffusively coupled systems}
Here we return to the motivating question raised in the Introduction.

Let $C$ be a proper cone in $\reals^n$, and $\{A_1,\dots, A_m\}\in {\cal L}(\reals^n)$ be QM for $C$. 
For all $i,j$ in $\{1,\dots, m\}$ with $i\neq j$, assume that $D_{ij}\in {\cal L}(\reals^n)$ and $D_{ij}=D_{ji}$.

We now define the coupled system on $(\reals^n)^m$:
\begin{eqnarray}
{\dot x}_1&=&A_1x_1 + \sum_{j\neq 1}D_{1j}(x_j-x_1)\label{coupled1}\\
{\dot x}_2&=&A_2x_2 + \sum_{j\neq 2}D_{2j}(x_j-x_2)\\
&\vdots& \nonumber \\
{\dot x}_m&=&A_mx_m + \sum_{j\neq m}D_{mj}(x_j-x_m)\label{coupledm}
\end{eqnarray}
Note that $C^m$ is a proper cone in $(\reals^n)^m$, and that its dual $(C^m)^*$ can be identified with $(C^*)^m$ thanks to the Riesz Representation Theorem.  
It is natural to ask when the proper cone $C^m$ in $(\reals^n)^m$ is a forward invariant set for  $(\ref{coupled1})-(\ref{coupledm})$. To answer this question, we introduce the following concept:

{\bf Definition} Let $C$ be a proper cone in $\reals^n$, and suppose that for all $i,j$ in $\{1,\dots, m\}$ with $i\neq j$, 
$D_{ij}\in {\cal L}(\reals^n)$ and $D_{ij}=D_{ji}$. We say that the collection $\{D_{ij}\}$ acts diffusively on $C$, provided that for all $i\neq j$:
\begin{enumerate}
\item 
$D_{ij}(C)\subseteq C$.
\item
Whenever $(x,\lambda)\in (\partial C,C^*)$ is such that $\lambda(x)=0$, then $\lambda(D_{ij}x)=0$.
\end{enumerate}
Note that for a given cone $C$, and fixed $m$, there always exist nontrivial families $\{D_{ij}\}$ that act diffusively on $C$. Indeed, if $D_{ij}=\alpha_{ij}I$ for some arbitrary $\alpha_{ij}=\alpha_{ji}\geq 0$, then the family $\{D_{ij}\}$ acts diffusively on $C$. When $C=\reals^n_+$, any family $\{D_{ij}\}$ consisting of diagonal matrices with non-negative entries, also acts diffusively on $C$.  In fact, it is not difficult to see that in this case, diagonal matrices with only non-negative entries are the only matrices that can belong to any family $\{D_{ij}\}$ that acts diffusively on $\reals^n_+$.

{\bf Notation}: For future reference, we let ${\cal D}_m$ be the (nonempty) set whose elements are all possible families 
$\{D_{ij}\}$ of linear operators on $\reals^n$ with $i,j$ in $\{1,\dots, m\}$ and $i \neq j$ such that 
$D_{ij}=D_{ji}$, that act diffusively on a given proper cone $C$ in $\reals^n$.

For example, when $C=\reals^n_+$, the set ${\cal D}_m$ is the set of diagonal matrices having only non-negative entries.

The following result remains valid even when the symmetry assumption $D_{ij}=D_{ji}$ is dropped, as it is never used in the proof.
\begin{stel}\label{coupled-invariance}
Let $C$ be a proper cone in $\reals^n$, $\{A_1,\dots, A_m\}\in {\cal L}(\reals^n)$ be QM for $C$, and $\{D_{ij}\} \in {\cal D}_m$.
%Suppose that for all $i\neq j$ in $\{1,\dots, m\}$, the family $\{D_{ij}\}\subseteq  {\cal L}(\reals^n)$ with $D_{ij}=D_{ji}$ acts diffusively on $C$. 
Then $C^m$ is a forward invariant set for $(\ref{coupled1})-(\ref{coupledm})$.
\end{stel}
\begin{proof}
We need to verify that the following linear operator on $(\reals^n)^m$:
\begin{equation}\label{coupled-matrix}
A_D=\begin{pmatrix}
A_1-\sum_{j\neq 1} D_{1j}& D_{12}&\dots &D_{1m}\\
D_{21}&A_2-\sum_{j\neq 2} D_{2j}& \dots  &D_{2m}\\
\vdots&\vdots &\ddots &\vdots \\
D_{m1}& D_{m2}&\dots & A_m- \sum_{j\neq m} D_{mj}
\end{pmatrix}
\end{equation}
is QM for $C^m$. Let $X=(x_1,\dots ,x_m)\in \partial C^m$ and $\Lambda=(\lambda_1,\dots, \lambda_m)\in (C^m)^*$ be such that $\Lambda(X)=0$. We are left with showing that $\Lambda (A_DX)\geq 0$.

Since $(C^m)^*=(C^*)^m$, it follows that $\lambda_i\in C^*$ for all $i=1,\dots, m$.  Thus, $\lambda_i(x_i)\geq 0$ for all $i=1,\dots, m$, but since $\Lambda(X)=\lambda_1(x_1)+\dots + \lambda_m(x_m)=0$, there follows that:
\begin{equation}\label{help1}
\lambda_i(x_i)=0\textrm{ for all } i=1,\dots ,m,
\end{equation}
and then Lemma $\ref{basics}$ implies that:
\begin{equation}\label{help2}
\textrm{ For all } i=1,\dots ,m:\textrm{ either } x_i\in \partial C,\textrm{ or if } x_i\notin \partial C \textrm{ then } \lambda_i=0.
\end{equation}
But $A_i$ is QM for $C$ for all $i=1,\dots ,m$, hence:
\begin{equation} \label{help3}
\lambda_i(A_ix_i)\geq 0\textrm{ for all }i=1,\dots ,m.
\end{equation}
Then, as $\{D_{ij}\}$ acts diffusively on $C$, and using $(\ref{help1})$, $(\ref{help2})$ and $(\ref{help3})$:
\begin{eqnarray*}
\Lambda(A_D X)&=&\lambda_1\left( (A_1-\sum_{j\neq 1}D_{1j})x_1 + D_{12}x_2+\dots +D_{1m}x_m\right)+\dots +\\
&&\lambda_m\left( D_{m1}x_1+D_{m2}x_2+\dots + (A_m-\sum_{j\neq m} D_{mj})x_m \right)\\
&=& \left(\sum_{i=1}^m \lambda_i(A_ix_i)\right) + \left(\sum_{i=1}^m \sum_{j\neq i}^m\lambda_i(D_{ij}x_j)\right)-0\\
&\geq& 0,
\end{eqnarray*}
which concludes the proof.
\end{proof}

\begin{stel}\label{coupled-suff}
Let $C$ be a proper cone in $\reals^n$, $\{A_1,\dots, A_m\}\in {\cal L}(\reals^n)$ be QM for $C$.
If $A_1,\dots, A_m$ have a common linear Lyapunov function on $C$, then for all $\{D_{ij}\} \in {\cal D}_m$ 
the zero steady state of $(\ref{coupled1})-(\ref{coupledm})$ is asymptotically stable with respect to all initial conditions in $(\reals^n)^m$.
\end{stel}
\begin{proof}
Fix $\{D_{ij}\}\in {\cal D}$. Let $\lambda \in C^*$ be a common linear Lyapunov function for $A_1,\dots, A_m$ on $C$, and define $\Lambda \in (C^m)^*=(C^*)^m$ as follows:
$$
\Lambda(X)=\lambda(x_1)+\dots +\lambda(x_m),\textrm{ for all } X=(x_1,\dots, x_m)\in C^m.
$$
We claim that $\Lambda$ is a linear Lyapunov function for system $(\ref{coupled1})-(\ref{coupledm})$ on $C^m$. Indeed,
$$
\Lambda(X) >0\textrm{ for all } X \in C^m\setminus \{0\},
$$
because when $X\in C^m\setminus \{0\}$, there exists at least one $x_i\in C\setminus \{0\}$ and for which $\lambda(x_i)>0$. 
Moreover, using the notation in $(\ref{coupled-matrix})$, we have that:
\begin{eqnarray*}
\Lambda(A_DX)&=&\lambda \left( (A_1-\sum_{j\neq 1}D_{1j})x_1 + D_{12}x_2+\dots +D_{1m}x_m\right)+\dots +\\
&&\lambda \left( D_{m1}x_1+D_{m2}x_2+\dots + (A_m-\sum_{j\neq m} D_{mj})x_m \right)\\
&=&\sum_{i=1}^m \lambda(A_ix_i) <0,\textrm{ for all } X\in C^m\setminus \{0\},
\end{eqnarray*}
where we used the symmetry $D_{ij}=D_{ji}$, and the fact that $\lambda$ is a common linear Lyapunov function on $C$ for $A_1,\dots, A_m$. This establishes the claim, and the conclusion now follows from Theorem $\ref{L-H}$.
\end{proof}
{\bf Example}: Let $C=\{x\in \reals^3 \, |\, (x_1^2+x_2^2)^{1/2}\leq x_3\}$ be the ice cream cone in $\reals^3$. Pick two distinct matrices $A_1$ and $A_2$ from the following family:
$$
\Bigg\{\begin{pmatrix}-\epsilon_1& -1& 0\\
1&-\epsilon_1& 0\\
0&0& -\epsilon_2
 \end{pmatrix},\textrm{ where } 0<\epsilon_2\leq \epsilon_1 \Bigg \}.
$$
We have seen that $A_1$ and $A_2$ are QM on $C$, and that they share a common Lyapunov function $\lambda(x)=x_3$ on $C$. Let $D_{12}=D_{21}=dI$, where $d\geq 0$ is arbitrary. We have seen that the family $\{D_{12}\}$ acts diffusively on $C$. Thus, by Theorem $\ref{coupled-suff}$, every solution of the system:
\begin{eqnarray*}
{\dot x}_1&=&A_1x_1+D_{12}(x_2-x_1)\\
{\dot x}_2&=&A_2x_2 + D_{21}(x_2-x_1)
\end{eqnarray*}
in $\reals^6$ converges to the zero steady state. 

Physically, we can think of two ice cream cones filled with water which are being emptied by gravity via their vertex in the origin. When 
there is no water exchange between the cones ($d=0$), the exponential rates at which the height of the water columns decrease is given by the two respective parameters $\epsilon_2$ of the matrices $A_1$ and $A_2$. The two parameters $\epsilon_1$ control the rate at which water particles spiral towards the symmetry axes of the cones. This happens with the same 
frequency $1$ in both cones. When a coupling term is present, ($d>0$) water is exchanged between the two cones at rate $d$, making them communicating vessels. The stability result above confirms among other things the intuition that the two cones will still be emptied eventually, independently of the rate of exchange of water between the cones. In fact, the total height of the two water columns is decreasing, and serves as a Lyapunov function for the coupled system.

{\bf Example}: We show that the converse of Theorem $\ref{coupled-suff}$ is not true.

Let $C=\reals^2_+$, and 
$$
A_1=\begin{pmatrix}
-1& 0\\
1&-1
\end{pmatrix},
A_2=\begin{pmatrix}
-1& 1\\ 
0& -1 \end{pmatrix}, D=\begin{pmatrix} d_1&0 \\ 0&d_2 \end{pmatrix}\textrm{ where } d_1,d_2\geq 0 \textrm{ are arbitrary}.
$$
We have seen that $A_1$ and $A_2$ are QM on $C$, but that they don't share a common linear Lyapunov function on $C$. We will show that the zero solution of
\begin{eqnarray}
{\dot x}_1&=&A_1x_1+D(x_2-x_1) \label{s1}\\
{\dot x}_2&=&A_2x_2+D(x_1-x_2) \label{s2}
\end{eqnarray}
is asymptotically stable in $(\reals^2)^2=\reals^4$ for all matrices $D$.

Note first that for every $D$, the matrix:
$$
{\cal A}(D)=
\begin{pmatrix}
A_1-D&D\\
D&A_2-D
\end{pmatrix}
$$
is QM on $(\reals^2_+)^2=\reals^4_+$ by Theorem $\ref{coupled-invariance}$. By the Perron-Frobenius Theorem \cite{berman} follows that ${\cal A}(D)$ has a real, principal eigenvalue 
$\lambda_p(D)$ (i.e. $|\lambda|<\lambda_p(D)$ for every eigenvalue $\lambda\neq \lambda_p(D)$ of ${\cal A}(D)$). Since $A_1$ and $A_2$ are Hurwitz it is clear that $\lambda_p(0)<0$. Moreover, $\lambda_p(D)$ is continuous in $D$. We claim that $\lambda_p(D)<0$ for all $D$. To see this, it suffices to show that the determinant of ${\cal A}(D)$ is positive for all $D$, and by using the fact that $A_1-D$ is invertible, we observe that:
\begin{eqnarray*}
\textrm{det}({\cal A}(D))&=&\textrm{det}\begin{pmatrix}
A_1-D&D\\
D&A_2-D
\end{pmatrix}\\
&=&\textrm{det}(A_1-D)\textrm{det}((A_2-D)-D(A_1-D)^{-1}D)
\end{eqnarray*}
Here we used the well-known identity that
$$
\textrm{det}\begin{pmatrix}P&Q\\R&S \end{pmatrix}=\textrm{det}(P)\textrm{det}(S-RP^{-1}Q),
$$
for all $n\times n$ matrices $P,Q,R$ and $S$ with invertible $P$, which is easily proved by observing that the following factorization always holds:
$$
\begin{pmatrix}P&Q\\R&S \end{pmatrix}=\begin{pmatrix}P&0\\R &I \end{pmatrix}\begin{pmatrix}I& P^{-1}Q \\ 0 & S-RP^{-1}Q \end{pmatrix}
$$
Therefore, 
\begin{eqnarray}
&&\textrm{det}({\cal A}(D))=\textrm{det}(A_1-D)\textrm{det}((A_2-D)-D(A_1-D)^{-1}D) \nonumber \\
&=&(1+d_1)(1+d_2)\textrm{det}\Bigg( \begin{pmatrix}-(1+d_1)&1 \\ 0& -(1+d_2) \end{pmatrix}-\begin{pmatrix} -\frac{d_1^2}{1+d_1}& 0\\
-\frac{d_1d_2}{(1+d_1)(1+d_2)} & -\frac{d_2^2}{(1+d_2)}\end{pmatrix}  \Bigg) \nonumber \\
&=&(1+d_1)(1+d_2) \textrm{det}\begin{pmatrix}\frac{d_1^2-(1+d_1)^2}{(1+d_1)}&1\\ \frac{d_1d_2}{(1+d_1)(1+d_2)}& \frac{d_2^2-(1+d_2)^2}{(1+d_2)} \end{pmatrix}\nonumber \\
&=&\textrm{det}\begin{pmatrix}
-(2d_1+1)& 1\\ d_1d_2 & -(2d_2+1)
\end{pmatrix} \nonumber \\
&=&
3d_1d_2+2(d_1+d_2)+1 \nonumber \\
&>&0,\textrm{ for all } d_1,d_2\geq 0. \label{contra}
\end{eqnarray}
Now if the zero solution of $(\ref{s1})-(\ref{s2})$ would not be asymptotically stable on $\reals^4$ for all $D$, then ${\cal A}(D)$ would not be Hurwitz for some matrix $D$.
Then there would exist some matrix ${\tilde D}$ such that 
$\lambda_p({\tilde D})=0$.
But  then $\textrm{det}({\cal A}({\tilde D}))=0$, contradicting $(\ref{contra})$.

\end{document}